\documentclass{amsart}

\let\oldInclude=\include
\def\include#1{\bgroup\def\clearpage{\relax}\oldInclude{#1}\egroup}

\usepackage{amsmath}
\usepackage{amsthm}
\usepackage{amsfonts}
\usepackage{amssymb}
\usepackage[]{amsmath, amsfonts, amssymb, epsfig}

\numberwithin{equation}{section}
\newtheorem {thm}{Theorem}[section]
\newtheorem {definition}[thm]{Definition}
\newtheorem {lem}[thm]{Lemma}
\newtheorem {prop} [thm] {Proposition}
\newtheorem {cor} [thm] {Corollary}
\newtheorem* {claim*} {Claim}

\begin{document}
	
\title {Mapping Properties of Quadrature Domains in Several Complex Variables} \thanks{Research supported by the NSF Analysis and Cyber-enabled Discovery and Innovation programs, grant DMS~1001701}

\author{Alan R. Legg}
\address{Indiana University Purdue University Fort Wayne, Ft. Wayne, IN}
\email{leggar01@ipfw.edu}

\begin{abstract}
	We make use of the Bergman kernel function to study quadrature domains for $L^2$ holomorphic functions of several variables. Emphasis is given to generalizing nice mapping properties of planar quadrature domains to the several-variable setting.
\end{abstract}
\maketitle

\section{Introduction}
 This article examines some properties of quadrature domains which can be analyzed from the viewpoint of complex analysis. Special attention will be given to the relationship between one and several variables; a major goal is to generalize elegant planar phenomena to several dimensions. 

The theory of quadrature domain is a few decades old, and has already experienced effective use in several fields of mathematics, including potential theory, Riemann surfaces, and complex analysis.  A classical quadrature domain, following Aharanov and Shapiro \cite{AS}, is a domain $\Omega \subset \mathbb{C}$ such that evaluation of integrals of functions in the class $AL^1(\Omega)$ (holomorphic functions which are integrable) is a finite computation involving point evaluations :  $\Omega$ is a quadrature domain if there exist finitely many points $z_1,\cdots z_K$ of $\Omega$, and finitely many complex constants $\{c_{jk}\}|_{j=1,k=1}^{J,K}$ such that, for every $f \in AL^1(\Omega)$,
\begin{equation}
\label{qi}
\int_{\Omega} f dA = \sum_{j\leq J,k \leq K} c_{jk} f^{(j)}(z_k).
\end{equation}

If $\Omega$ is such a domain, the relation (\ref{qi}) is called its  `quadrature identity,' and the points $z_k$ are called `quadrature nodes.' The definition can be modified by changing the test class of functions for which the quadrature identity must hold; for instance, we could insist on a formula valid for certain harmonic functions, or, relevant to the current article, functions in the Bergman space.  Aharanov and Shapiro relate in the above-cited article that such domains arose in solutions to extremal problems, but proved interesting on their own.

Aharanov and Shapiro originally found a satisfying relationship between the quadrature identity of a quadrature domain and the so-called `Schwarz function' of a domain. Developing this line of thought led to a nice list of properties enjoyed by quadrature domains; in \cite{AS},  they showed among other things that a bounded simply connected quadrature domain must be a rational image of the unit disc with poles outside the closed unit disc, and that the boundary of any quadrature domain must be contained in some algebraic curve.

 Following these developments, Avci was able in his doctoral dissertation \cite{Avci} to incorporate conformal mapping and the Bergman kernel into the study of quadrature identities. To point out a motivational example, he showed that the rational biholomorphic image of a quadrature domain with poles outside the closure, is again a quadrature domain (generalizing the disc/simply connected case of Aharanov and Shapiro).

 The theory then developed along various lines. Gustafsson \cite{Gustafsson} used Riemann surfaces called Schottky doubles together with meromorphic differentials to expand and refine results of Aharanov and Shapiro. For example, he showed that the boundary of a quadrature domain is a whole algebraic curve except possibly finitely many points, and that nearby bounded domains with good boundary there exist biholomorphically equivalent quadrature domains.

    Eventually, it was realized that quadrature domains have a connection to such topics as fluid dynamics, Laplacian growth, free boundaries, Hele-Shaw cells, linear algebra, subnormal operators, and exponential transforms. Fascinating though these connections are, we will not have more to say about them, other than to suggest to the reader the expository article \cite{GS} and the whole volume \cite{Proc}, which offer an account of the theory of quadrature domains as of $2005$.

  From the standpoint of complex analysis, Bell packaged the theory of quadrature domains for square-integrable analytic functions in the language of the Bergman kernel and potential theory, synthesizing the efforts of Aharanov and Shapiro, Avci, and Gustafsson \cite{Bell3, Bell4, Bell5, Bell6, BGS}. He noticed that on quadrature domains, entities like Ahlfors maps, the Bergman kernel, the Szeg\H{o} kernel, and the harmonic measure functions are algebraic, and take very simple forms on the boundary \cite{Bell6}. From here he was able to describe quadrature identities and conformal maps into quadrature domains in terms of relations involving the Bergman kernel function.

This approach also led to glimpses of a theory of quadrature domains in several complex variables \cite{Bell5}, which is the topic of this article.  For quadrature domains whose quadrature identities are satisfied for real-valued harmonic functions, the several-dimensional approach is managed by writing quadrature identities in terms of an object called the `Schwarz Potential', and using real potential theory: see \cite{Shapiro2} for some of the introductory ideas.  But the question of what quadrature domains for holomorphic functions of several complex variables look like has been largely unexplored (to my knowledge, \cite{Bell5, HV} are the two references so far). This possible development of the theory is an open question of Sakai \cite{Sakai}.  This article employs complex analysis and the Bergman kernel to analyze mapping properties of quadrature domains in several complex variables

 We will first give a consideration of product domains and quadrature domains. Then we look at some biholomorphic mapping properties of what Bell calls `Quadrature domains with a capital Q', and `Bergman coordinates' \cite{Bell3, Bell5}. A few counterexamples against planar behaviors in many variables are offered. After this, we establish that smooth bounded convex domains are always biholomorphic to a quadrature domain. We end with a return to the plane, where homotopies through quadrature domains are exhibited. 

\section{Product Domains}
In this article, we will be concerned solely with quadrature domains in $\mathbb{C}^n$ whose quadrature identity is valid for functions in the Bergman space $H^2$ of square-integrable holomorphic functions.  That is to say, $\Omega \subset \mathbb{C}^n$ will be a quadrature domain if there exist finitely many constants $c_{i\alpha}$ and points $z_i$ such that for all $f \in H^2(\Omega)$, 

\[\int_\Omega f(z)dz= \sum_{i,\alpha} c_{i\alpha} \frac{\partial ^\alpha f}{\partial z^\alpha }(z_i),   \]
where $\alpha$ are multiindices.

We begin our study of quadrature domains for holomorphic functions of several complex variables by considering an intuitive case: that of a domain in $\mathbb{C}^n$ which is the Cartesian product of bounded planar quadrature domains.  These domains will retain some of the properties of quadrature domains in $\mathbb{C}$, and so are a natural starting point.

 Recall that given a domain $\Omega$, the Bergman kernel function $K(z,w)$ defined on $\Omega \times \Omega$ is the function guaranteed to reproduce functions $f\in H^2(\Omega)$ under inner products: $\int_\Omega  f(w)K(z,w) dw = f(z).$ It is holomorphic in the first argument and antiholomorphic in the second.  Inner products taken against antiholomorphic derivatives of $K$ in the antiholomorphic variable result in evaluation of derivatives of Bergman-space functions.

Since quadrature identities also invoke the evaluation of derivatives of Bergman-space functions, we recall the following definition:

\begin{definition} The Bergman span of a domain $\Omega \subset \mathbb{C}^n$ is the complex linear span of the functions $\frac{\partial ^\alpha}{\partial \bar{w}^\alpha}K(z,w)|_{w=w_0}$, where $K$ is the Bergman kernel of $\Omega$, $\alpha$ ranges over multiindices, and $w_0$ ranges over $\Omega$.
\end{definition}

Thus we see that, after restricting ourselves to domains of finite volume, $\Omega$ is a quadrature domain for $H^2$ functions if and only if the function $1$ is in the Bergman span of $\Omega$.

Using the fact that for a product domain $\Omega=\Omega_1 \times \Omega_2 \times \cdots \Omega_n$, the Bergman kernel satisfies \begin{equation}
\label{eq1}
K_\Omega(z,w)=K_{\Omega_1}(z_1,w_1) K_{\Omega_2}(z_2,w_2) \cdots K_{\Omega_n}(z_n,w_n),
\end{equation}
 we can draw a conclusion about the possibility of quadrature domains being product domains.

\begin{prop}
	\label{prop2}
Let $\Omega=\prod_{j=1}^{n}\Omega_j$, a set of finite volume, be a cross product of planar domains of finite area.  Then $\Omega$ is a quadrature domain for $H^2$ functions if and only if each of the domains $\Omega_j$ is a quadrature domain for $H^2$ functions.

\begin{proof}
	We first suppose that $\Omega$ is a quadrature domain for $H^2$. Since the Bergman span of $\Omega$ then contains $1$, there are multi-indices $\alpha$, points $w^{(j)}$, and constants $c_{\alpha,j}$ such that
	\[\sum_{j \le J, |\alpha| \leq N}c_{\alpha,j}\frac{\partial^\alpha}{\partial \bar{w}^\alpha} K_\Omega(z,w)|_{w=w^{(j)}}=1.\]
	And in light of (\ref{eq1}), this expression may be expanded to
	\[\sum_{j \le J, |\alpha| \le N}c_{\alpha,j} \frac{\partial^{\alpha_1}}{\partial \bar{w_1}^{\alpha_1}}K_1(z_1,w_1)|_{w_1=w^{(j)}_1} \prod_{k=2}^{n}\frac{\partial^{\alpha_k}}{\partial \bar{w_k}^{\alpha_k}}K_k(z_k,w_k)|_{w_k=w^{(j)}_k}=1.   \]
	From here, let $(z_2,z_3,\cdots,z_n)$ be fixed in $\prod_{j=2}^{n}\Omega_j$.  After collecting like terms, for some new constants $C_{\alpha_1,j}$,
	\[\sum_{j \le J, \alpha_1 \le N} C_{\alpha_1,j} \frac{\partial^{\alpha_1}}{\partial \bar{w_1}^{\alpha_1}}K_1(z_1,w_1)|_{w_1=w^{(j)}_1} =1.  \]
	That means $1$ is in the Bergman span of $\Omega_1$, so $\Omega_1$ is a quadrature domain.  The other $\Omega_j$ are quadrature domains in the same way.
	
	For the converse implication, assume that each $\Omega_j$ is a quadrature domain.  Then for each $j$, the function $1$ is in the Bergman span of $\Omega_j$, and we can choose for each $j$ an appropriate linear combination of derivatives of the Bergman kernel in the antiholomorphic variable which add to $1$. Let $\omega_j$ denote the finite set of all of the quadrature nodes of $\Omega_j$, and let $\omega$ be the product of the $\omega_j$: $\omega=\omega_1 \times \omega_2 \times \cdots \times \omega_n$. Denumerate $\omega$ as a list of points $\omega=\{\omega^{(l)}\}_{l \le L}$. The notation $\omega^{(l)}_j \in \Omega_j$ will signify the $j^{th}$ coordinate of the point $\omega^{(l)}$.
	
	Having for each $j$ an element of the Bergman span of $\Omega_j$ equal to $1$, multiply them all together, and observe that for some positive integer $N$ and constants $c^{(l)}_{\alpha, j}$,
	\begin{equation}
	\label{eq2}
	\prod_{j=1}^{n} \sum_{|\alpha| \le N,l \le L} c^{(l)}_{\alpha,j} \frac{\partial ^{\alpha_j}}{\partial \bar{w}_j^{\alpha_j}}K_j(z_j,w_j)|_{w_j=\omega_j^{(l)}}=1^n=1.
	\end{equation}
	Keeping in mind (\ref{eq1}), we write for any multiindex $\alpha$:
	\[\prod_{j=1}^{n} \frac{\partial ^{\alpha_j}}{\partial \bar{w_j}^{\alpha_j}}K_j(z_j,w_j) =\frac{\partial ^\alpha}{\partial \bar{w}^\alpha} K_\Omega (z,w).\]
	Upon distributing the multiplication in (\ref{eq2}), regrouping, and simplifying, we have for some new constants $C_{\alpha, l}$:
	\[\sum_{|\alpha| \le nN, l \le L} C_{\alpha, l} \frac{\partial ^\alpha}{\partial \bar{w}^\alpha}K_\Omega (z,w)|_{w=\omega^{(l)}}=1.   \]
	
	So $1$ is in the Bergman span of $\Omega$, and $\Omega$ is a quadrature domain for $H^2$.
\end{proof}
\end{prop}

In the plane, finite-area quadrature domains for $H^2$ have a better Bergman span inclusion property than just for the function $1$: every holomorphic polynomial on such a domain is in the Bergman span \cite{Bell5}.  This a priori stronger requirement, as we will see later, is not upheld in the case of several complex dimensions.  In \cite{Bell5}, Bell called a quadrature domain in $\mathbb{C}^n$ which does have the property that all holomorphic polynomials reside in its Bergman span, a Quadrature Domain (with a capital `Q'). He also asked whether such quadrature domains exhibit strong mapping properties as in the planar case.  We will refer to such a domain as a `QDP'(short for `Quadrature Domain with Polynomials in the Bergman Span.' We will touch on mapping questions for QDP's in what follows, but for now, we show that products of smooth bounded planar quadrature domains are actually examples of QDP's in several dimensions.

\begin{prop}
	\label{QDcross}
	If $\Omega_j$, $j=1,\thinspace 2, \cdots, \thinspace n$, are smooth finite-area domains in the plane whose cross product $\Omega=\prod_{j=1}^{n}\Omega_j$ is of finite volume, then $\Omega$ is a QDP if and only if each $\Omega_j$ is a QDP.
	
	\begin{proof}
	Assuming that $\Omega$ is a QDP, we appeal to the previous proposition to conclude that each $\Omega_j$ is a quadrature domain of finite area in the plane.  Any such quadrature domain is automatically a QDP, as has been noted.
	
	For the reverse implication, take each $\Omega_j$ to be a QDP.  The method of proof from the previous proposition will show that any holomorphic monomial $z^\alpha$ is in the Bergman span of $\Omega$.  We need only choose for each $j$ a member of the Bergman span of $\Omega_j$ equal to $z_j^{\alpha_j}$, and then multiply all of them together. We obtain (\ref{eq2}) once more, except that the right hand side has become $z_1^{\alpha_1} z_2^{\alpha_2}\cdots z_n^{\alpha_n}=z^\alpha.$ Simplifying the left hand side shows that a member of the Bergman span of $\Omega$ is equal to $z^\alpha$.  Having verified the presence of each monomial in the Bergman span, we are done by linearity.
	\end{proof}
\end{prop}

One of the useful mapping properties of planar quadrature domains is that the image of a biholomorphic map is a quadrature domain if and only if the derivative of the map appears in the Bergman span \cite{Avci, Bell4}; replacing the derivative with the complex Jacobian determinant generalizes the fact to $\mathbb{C}^n$ \cite{Bell5}. The planar version quickly leads to a theorem for identifying simply connected bounded planar quadrature domains. We reproduce these well-known facts here. Statements and proofs are in \cite{AS, Avci, Bell4}.

\begin{thm}
 \label{Jacobthm}
 Let $f$ be a biholomorphic mapping from $\Omega \subset \mathbb{C}^n$ to $V \subset \mathbb{C}^n$, with $\Omega$ and $V$ of finite volume, $n \geq 1$.  Then $V$ is a quadrature domain if and only if the complex Jacobian determinant of $f$ is contained in the Bergman span of $\Omega$.
 \end{thm}

\begin{thm}
	\label{thm1}
 A simply connected bounded domain $\Omega \subset \mathbb{C}$ is a quadrature domain if and only if $\Omega$ is the image of the unit disc under a rational biholomorphism with all poles outside the closure of the unit disc.
\end{thm}

Due to the lack of a Riemann mapping theorem in more than one variable, there is no hope of generalizing Theorem \ref{thm1} to many variables by using, for example, the polydisc or ball in place of the unit disc.  However, if we restrict our focus to products of simply connected domains, we can generalize directly. We know already that product quadrature domains are products of quadrature domains, so the following proposition has to do with the automorphism group of the polydisc.

\begin{prop}
	\label{sccp}
The product $\Omega \subset \mathbb{C}^n$ of $n$ simply connected bounded planar domains is a quadrature domain if and only if $\Omega$ is the image of a rational biholomorphism from the unit polydisc, with all singularities off the closure of the unit polydisc.
\begin{proof}
	For the forward implication, note that if $\Omega=\prod_{j=1}^{n} \Omega_j$ is a quadrature domain, then by Proposition \ref{prop2}, each $\Omega_j$ is a bounded quadrature domain. Then by Theorem \ref{thm1}, each $\Omega_j$ is the image of a rational biholomorphism $f_j$ from the unit disc, with singularities off the closure of the unit disc.  A rational biholomorphism from the polydisc to $\Omega$ is thus $f(z_1,z_2,\cdots,z_n)=(f_1(z_1), f_2(z_2),\cdots, f_n(z_n))$.  Since some coordinate of any boundary point of the unit polydisc lies on the unit circle, we see that $f$ has no singular points on the boundary of the unit polydisc.
	
	For the converse, assume that the product domain $\Omega$ is the image of a rational biholomorphism $R$ of the unit polydisc with no singularities on the closure of the unit polydisc. By the Riemann mapping theorem, there is a biholomorphism from the unit disc to $\Omega_j$, say $f_j$. The map $f(z)=(f_1(z_1),f_2(z_2),\cdots,f_n(z_n))$ is thus a biholomorphism from the unit polydisc to $\Omega$.
	
	Letting $R^{-1}$ be the inverse of $R$, we use a composition and consider the map $R^{-1} \circ f$ from the unit polydisc to itself.  We know from the theory of several complex variables (for instance \cite{Kranz}) that the only automorphisms of the unit polydisc are of the form $\mu \circ \sigma$, where $\mu(z)=(\mu_1(z_1),\mu_2(z_2),\cdots,\mu_n(z_n))$, each $\mu_j$ being a M\"{o}bius transform of the unit disc, and $\sigma$ being a linear transformation on $\mathbb{C}^n$ which permutes coordinates; that is, $\sigma(z_1,z_2,\cdots,z_n)=(z_{s(1)},z_{s(2)},\cdots, z_{s(n)})$, for a permutation $s$ of $n$ letters. Hence, for an appropriate $\mu$ and $\sigma$, $R^{-1} \circ f = \mu \circ \sigma$. 
	
	 Precomposition by $R$ results in the equality $f=R \circ \mu \circ \sigma$. But each map on the right is rational, and so $f$ is rational, which means each $f_j$ is rational.  Since each of $R$, $\mu$, $\sigma$ is nonsingular on the closure of the polydisc, we conclude that each $f_j$ is also. Hence each $\Omega_j$ is a rational image of the unit disc under a map with poles off the closed unit disc, and by Theorem \ref{thm1}, each $\Omega_j$ is a bounded quadrature domain. Finally, Proposition \ref{prop2} ensures that $\Omega$ is a quadrature domain.
\end{proof}
\end{prop}
Not every rational image of the polydisc is a quadrature domain, as we will see later. The proposition covers the case in which the image is already assumed to be a product.  
 
 This introduction to quadrature domains by means of products has offered us a few valuable insights-first of all, quadrature domains and QDP's exist plentifully in several complex variables, and there are mapping properties to be found which might be similar to those found in the plane. The next section will explore some more mapping properties.

\section{Mapping Properties}

In this section we prove mapping properties of quadrature domains in $\mathbb{C}^n$, and in the spirit of \cite{Bell5}, will be especially concerned with QDP's.  Just as the restriction to product domains provided the basis necessary to prove Proposition \ref{sccp}, in this section the polynomial structure contained within the Bergman span of the domains will play an important role.

The crux of the matter is the following result, whose corollary will be an analogue of Theorem \ref{thm1}.  Recall that if $\Omega$ and $V$ are domains of finite volume and biholomorphic under $f: \Omega \rightarrow V$, there is a unitary isomorphism of the Bergman spaces $H^2(\Omega)$ and $H^2(V)$;  $\Lambda_1 : H^2(V) \rightarrow H^2(\Omega)$ given by $\Lambda_1(g)=u \cdot g \circ f$, $g\in H^2(V)$. Here $u$ is the complex Jacobian determinant of the map $f$.  Letting $F$ be the inverse mapping to $f$, and $U$ the complex Jacobian determinant of $F$, we have the inverse transformation $\Lambda_2$, defined by $\Lambda_2(h)=U \cdot h \circ F$ for $h\in H^2(\Omega)$ (for instance, see \cite{Bergman, BellBook}).

While $\Lambda_1$, $\Lambda_2$ provide a relationship between the Bergman spaces of $\Omega$ and $V$, it turns they even give a one-to-one correspondence between the Bergman spans of $\Omega$ and $V$. This lemma is hinted at in \cite{Bell5}, but not spelled out.

\begin{lem}
	\label{lem1}
Suppose that $\Omega$ is a domain of finite volume in $\mathbb{C}^n$, that $f$ is a biholomorphism defined on $\Omega$, and that $V=f(\Omega)$, also of finite volume.  Then, given a function $g \in H^2(V)$, $\Lambda_1(g)$ is in the Bergman span of $\Omega$ if and only if $g$ is in the Bergman span of $V$.

\begin{proof}
	For the forward direction, assume that $\Lambda_1(g)=u\cdot g \circ f$ is in the Bergman span of $\Omega$.  This means that for functions in $H^2(\Omega)$, taking an inner product against $u \cdot g \circ f$ results in a finite linear combination of point evaluations of derivatives. Thus for some points $\omega_j$ of $\Omega$, and positive integer $N$, and constants $c_{\alpha, j}$, we have for any $\varphi \in H^2(\Omega)$,
	\begin{equation}
	\label{bkr}
	\langle \varphi,\thinspace u \cdot g \circ f \rangle=\sum_{j \le J, |\alpha| \le N} c_{\alpha,j} \frac{\partial ^\alpha \varphi}{\partial z^\alpha}|_{z=\omega_j}.
	\end{equation}
	Letting $h$ denote an arbitrary function in $H^2(V)$, we will utilize this formula for the convenient choice of $\varphi=\Lambda_1(h)$.  So we compute, starting on the $V$ side, and employing (\ref{bkr}):
	\[\langle h,\thinspace g \rangle = \langle \Lambda_1(h), \thinspace \Lambda_1(g) \rangle= \sum_{j \le J, |\alpha| \le N} c_{\alpha, j} \frac{\partial ^\alpha (u \cdot h \circ f)}{\partial z^\alpha}|_{z=\omega_j}.  \]
	It is now a routine matter to calculate the terms of the sum using the Leibniz and chain rules repeatedly. Since $f$ is a fixed function, the evaluations involving $f$ and $u$ at the various $\omega_j$ do not depend on $h$ in any way, and so we can see that:
	
	\begin{equation}
	\label{thegreatblah}
\langle h, \thinspace g \rangle = \sum_{j\le J, |\alpha| \le N} C_{\alpha, j} \frac{\partial ^\alpha h}{\partial \zeta^\alpha}|_{\zeta=w_j},
	\end{equation}
		for some new constants $C_{\alpha, j}$, where we have denoted by $\zeta$ the coordinate on $V$, and have set $w_j=f(\omega_j)$.
		
	The sum on the right is none other than the following inner product:
	
\[\langle h, \thinspace \sum_{j \le J, |\alpha| \le N} C_{\alpha, j} \frac{\partial^\alpha K(\zeta, w)}{\partial \bar{w}^\alpha}|_{w=w_j} \rangle. \]
Thus, for every $h\in H^2(V)$, we have the equality
 \[ \langle h, \thinspace g \rangle = \langle h, \thinspace \sum_{j \le J, |\alpha| \le N} C_{\alpha, j} \frac{\partial ^\alpha K(\zeta, w)}{\partial \bar{w}^\alpha}|_{w=w_j} \rangle, \]
 and since $g$ and the Bergman span element on the right are both themselves in $H^2(V)$, we conclude that $g$ is identical to that element of the Bergman span.

 On the other hand, if we first assume that $g$ is in the Bergman span of $V$, then we note that for any $h\in H^2(\Omega)$,
 \[\langle h, \thinspace \Lambda_1(g) \rangle = \langle \Lambda_2(h), \thinspace g \rangle, \]
 since $\Lambda_1,\Lambda_2$ are unitary. The proof now proceeds just as above, using the fact that an inner product against $g$ is a finite linear combination of evaluations of derivatives, together with the Leibniz and chain rules on $\Lambda_2(h)$, to arrive at the conclusion that the original inner product $\langle h, \thinspace \Lambda_1(g) \rangle $ results in a finite linear combination of evaluations of derivatives of $h$. And since $h$ was arbitrary, this will imply that $\Lambda_1(g)$ must be identical to the Bergman span element of $\Omega$ which reproduces the same linear combination of evaluations of derivatives.
\end{proof}
\end{lem}

Given the lemma, we can now assert that if $f : \Omega \rightarrow V$ is a biholomorphism, then letting $\zeta$ denote coordinates on $V$, it follows that the monomial $\zeta^\alpha$ is in the Bergman span of $V$ if and only if $\Lambda_1(\zeta^\alpha)$ is in the Bergman span of $\Omega$. However, $\Lambda_1(\zeta^\alpha)$ is simply the function $u \cdot f^\alpha$, where $u$ is the complex Jacobian determinant of $f$.  Appealing to linearity, we have proved the following theorem, which is a fortification of Theorem \ref{Jacobthm}. A one-dimensional version of this theorem is hinted at in \cite{Bell5}. 

\begin{thm}
	\label{QDthm}
Let $\Omega$ and $V$ be domains of finite volume in $\mathbb{C}^n$, and let $f: \Omega \rightarrow V$ be a biholomorphism between them. Then, $V$ is a QDP if and only if for each multi-index $\alpha$, the function $uf^\alpha$ is in the Bergman span of $\Omega$, where $u$ is the complex Jacobian determinant of $f$.
\end{thm}

Since many different mappings may have the same Jacobian determinant, it seems that mapping into a QDP is indeed more restrictive than mapping into a generic quadrature domain. This is in line with the intuition of \cite{Bell5}.

Theorem \ref{QDthm} can be used to relate the Bergman kernel to mappings into QDP's.  These results will be phrased in the terminology of `global Bergman coordinates,' as they appear in \cite{Bell3, Bell5}.

\begin{definition}
A global Bergman coordinate (or just `Bergman coordinate') of a domain $\Omega \subset \mathbb{C}^n$ is a biholomorphic mapping defined on $\Omega$, each of whose component functions is a quotient of elements of the Bergman span of $\Omega$.
\end{definition}

Such mappings are desirable in the plane because they relate to a Riemann surface interpretation of quadrature domains from which many strong properties follow, including algebraicity and the fact that planar Bergman coordinates always map into quadrature domains \cite{Bell3}.  Localized versions of Bergman coordinates were used in several variables by Bell and Ligocka to establish boundary regularity of biholomorphisms between certain domains \cite{BL}.  An introductory picture of the situation regarding Bergman coordinates and quadrature domains will be worked toward in the remainder of this chapter, and in part of the next.

Our first result in this direction is that Bergman coordinates do have some relation to mappings of QDP's. But the relationship is not as strong as in the plane, which we will see later.

\begin{prop}
	\label{prop3}
Let $f: \Omega \rightarrow V$ be a biholomorphism between domains of finite volume in $\mathbb{C}^n$, and assume that $V$ is a QDP.  Then $f$ is a Bergman coordinate.

\begin{proof}
	This is a consequence of Theorem \ref{QDthm}. Write $f=(f_1(z), f_2(z), \cdots, f_n(z))$, where $z$ is the coordinate on $\Omega$.  Since $\zeta^\alpha$ is in the Bergman span of $V$ for each multiindex $\alpha$, where $\zeta$ is the coordinate on $V$, we see that for each $j=1,2,\cdots, n$, $u\cdot f_j$ is in the Bergman span of $\Omega,$ as is $u$. Thus for each $j$, $f_j=\frac{u\cdot f_j}{u}$ is a quotient of Bergman span elements.
\end{proof}
\end{prop}

 As an example of the utility of this result, the following corollary offers more information about biholomorphic mapping between QDP's when one of them is a product domain.

\begin{cor}
	\label{cor1}
	Let $\Omega$ be a quadrature domain of finite volume which is a cross product of $n$ planar domains of finite area.  Then, if $V$ is a QDP and $f: \Omega \rightarrow V$ is a biholomorphism, then $f$ is an algebraic mapping.

\begin{proof}
First, let $\Omega_j$, $j=1,2,\cdots, n$ be the domains of which $\Omega$ is the product. By the planar theory of quadrature domains, the Bergman kernel function of each $\Omega_j$ is an algebraic function of $z_j$ \cite{Bell6}.  Since the Bergman kernel of $\Omega$ is simply the product of the Bergman kernels of the $\Omega_j$, it follows that the Bergman kernel of $\Omega$ is an algebraic function.  Thus every element of the Bergman span of $\Omega$ is algebraic, and the proof is finished.
\end{proof}
\end{cor}

 We can readily admit some variations in the hypotheses; for example, if $\Omega$ is assumed to have a rational Bergman kernel, then any biholomorphic map from $\Omega$ to a QDP will be rational. Hence we have another analogue of Theorem \ref{thm1}:

 \begin{cor}
 Let $\mathbb{B}^n$ be the unit ball of $\mathbb{C}^n$. Then, if $f$ is a biholomorphic mapping $f: \mathbb{B}^n \rightarrow V$, with $f(\mathbb{B}^n)=V$ being a QDP, then $f$ must be a rational map.

 \begin{proof}
 We know $f$ must be a Bergman coordinate, and that the Bergman kernel of the unit ball is rational. Thus each component of $f$ is a quotient of rational functions.
 \end{proof}
 \end{cor}
Not every rational biholomorphic map on the ball is guaranteed to lead to a quadrature domain. In higher dimensions only half of Theorem \ref{thm1} survives.

The relationship between maps into QDP's and the Bergman kernel can also be employed in regard to circular domains, where the Bergman kernel is particularly well-behaved. Recall that a domain $\Omega \subset \mathbb{C}^n$ is called `circular' if whenever $z \in \Omega$ and $k\in \mathbb{C},|k|=1$, it follows that $kz \in \Omega$. In \cite{Bell1, BellProc, Mittag-Leffler} it is proved that not just the whole Bergman span, but even the so-called `Bergman span associated to the point $0$' of a circular domain, contains all holomorphic polynomials in a structured way.

 \begin{definition}
 If $\Omega$ is a domain in $\mathbb{C}^n$, and $w_0 \in \Omega$, then the Bergman span associated to the point $w_0$, or simply the `Bergman span at $w_0$' is the complex linear span of the functions $\frac{\partial^\alpha K(z,w)}{\partial \bar{w}^\alpha}|_{w=w_0}$ with $\alpha$ varying over multi-indices, but $w_0$ held fixed.
 \end{definition}

  On a bounded circular domain that contains the origin, given a multiindex $\alpha$, there exists a homogeneous holomorphic polynomial $P_\alpha$ of degree $|\alpha|$ such that $\frac{\partial^\alpha K(z,w)}{\partial \bar{w}^\alpha}|_{w=0}=P_\alpha$. Evaluation at $0$ of the $\alpha$ derivative of an $H^2$ function, is equivalent to taking an inner product against a homogeneous holomorphic polynomial.  Even more, orthonormality ensures that on a bounded circular domain containing the origin, a holomorphic function $g$ is a holomorphic polynomial if and only if $g$ is orthogonal to all holomorphic monomials of sufficiently high degree \cite{BellProc}.

  In preparing for the next theorem, let us agree to call a domain $\Omega$ a $1$-point QDP if there exists a point $w_0 \in \Omega$ such that the set of all holomorphic polynomials is contained in the Bergman span at $w_0$. We will say $\Omega$ is a $1$-point QDP `at $w_0$'.  In the literature it is sometimes implied (e.g. \cite{Bell5}) that the quadrature identity of a $1$-point quadrature domain should not involve derivatives, which in the plane for example would force the disc to be the only $1$-point quadrature domain.  Our definition here allows the function $1$ to be expressed as a Bergman span element involving derivatives. In the plane this allows certain lemniscates to be included in the definition \cite{Shapiro2, AS}. The proof of the following theorem follows the logic of \cite{BellProc}, and can be viewed as a generalization of Cartan's Uniqueness theorem governing maps between circular domains which fix the origin.

  \begin{thm}
  	\label{CircThm}
  Let $\Omega \subset \mathbb{C}^n$ be a bounded circular domain which contains $0$. Let $V$ be a domain of finite volume.  If $f: \Omega \rightarrow V$ is a biholomorphism such that $f(0)=w_0$, then $V$ is a $1$-point QDP at $w_0$ if and only if $f$ is a polynomial mapping.

  \begin{proof}
  	By translation, there is no loss of generality in assuming that $w_0=0$.
  	
  	If $f$ is a polynomial mapping, then it is clear that for every multiindex $\alpha$, $u\cdot f^\alpha$ is a holomorphic polynomial, where $u$ is the complex Jacobian determinant of $f$. Since every polynomial is in the Bergman span at $0$ for a bounded circular domain $\Omega$ containing the origin, $u \cdot f^\alpha$ is in the Bergman span.  By Lemma \ref{lem1}, this implies that $\zeta^\alpha$ is in the Bergman span of $V$. A perusal of Lemma \ref{lem1} will show that in fact $\zeta^\alpha$ is in the Bergman span at $0$ for $V$ ((\ref{thegreatblah}) and the words after).  So all the holomorphic monomials are in the Bergman span at $0$ for $V$, and $V$ is a $1$-point QDP at $0$.
  	
  	Conversely, assuming that $V$ is a $1$-point QDP at $f(0)=0$, we will use the fact from \cite{BellProc} pointed out before the theorem that to test whether a function is a holomorphic polynomial on $\Omega$, we need only take its inner product against monomials of high order.  For fixed multiindex $\alpha$, the inner product against $z^\beta$, where $z$ is the coordinate on $\Omega$, may be calculated as follows:
  	\begin{equation}
  	\label{inner1}
  	\langle z^\beta, \thinspace u\cdot f^\alpha \rangle = \langle U \cdot F^\beta, \thinspace \zeta^\alpha \rangle=\sum_{|\gamma| \le |\alpha|} c_\gamma \frac{\partial ^\gamma }{\partial \zeta^\gamma}(U \cdot F^\beta)|_{\zeta=0},
  	\end{equation}
  	where $F$ is the inverse mapping to $f$, $U$ is the complex Jacobian determinant of $F$, $\zeta$ is the coordinate on $V$, the $\gamma$ are multiindices, and the $c_\gamma$ are constants. The relation (\ref{inner1}) holds because $\zeta^\alpha$ is presumed to be in the Bergman span at $0 \in V$.
  	
  	Appealing to the chain and Leibniz rules on the right side of (\ref{inner1}), it is the case that for $|\beta|$ sufficiently large (larger than $|\alpha|$)  , every term on the right includes as a factor some component function of $F$. Since $F(0)=0$, the whole inner product (\ref{inner1}) is therefore $0$.  Hence $u \cdot f^\alpha$ is orthogonal to all homogeneous holomorphic polynomials of sufficiently high order on $\Omega$, which implies that $u \cdot f^\alpha$ is itself a holomorphic polynomial.  Since the ring of holomorphic polynomials is a unique factorization domain and $u \cdot f^\alpha$ is polynomial for arbitrary $\alpha$, it follows that $f$ itself is a polynomial mapping.  The details of this last algebraic step are deferred to a lemma. (In \cite{BellProc}, this lemma is mentioned without a proof, but for completeness we include one here.)
  	\end{proof}
  \end{thm}

  \begin{lem}
  	\label{lem2}
  	Let $u$ be a holomorphic function on a domain $\Omega \subset \mathbb{C}^n$, which is the complex Jacobian determinant of $f$, a biholomorphic mapping defined on $\Omega$.  If for all multiindices $\alpha$, $uf^\alpha$ is a holomorphic polynomial, then $f$ is a polynomial mapping.
  	\begin{proof}
  		Consider that $uf_1^j$ is a polynomial for each $j$. In particular, $u$ is a polynomial by setting $j=0$, and $f_1$ is rational by the division $f_1=\frac{uf_1}{u}$.  Since the ring of polynomials over $\mathbb{C}$ in several variables is a unique factorization domain, we can express $f_1$ as a fraction in lowest terms. Let $f_1=\frac{p}{q}$, with $p$ and $q$ holomorphic polynomials which do not share any irreducible factors.  For each $j$, let $uf_1^j=P_j$.  We now substitute for $f_1$ so that $u \cdot \frac{p^j}{q^j}=P_j$. This means that $up^j=q^jP_j$. But since $p$ and $q$ share no irreducible factors, if we envision breaking each side into irreducible factors, all the factors of $p^j$ must therefore appear in the factorization of $P_j$. Say that $P_j=p^jQ_j$, $Q_j$ a polynomial.  After cancelling $p^j$, we have arrived at $u=q^jQ_j$. Since this holds for all $j$, by counting degrees our only possibility is that $deg(q)=0$, or else $Q_j=0$ for all $j$. If $Q_j=0$, then $u=0$ and $f$ fails to be biholomorphic. This means that we must have $deg(q)=0$, and so $f_1$ is a polynomial.  Similarly for $f_2, \cdots, f_n$.
  	\end{proof}
  \end{lem}
  
 We have seen that mappings to QDP's can be expressed with the Bergman kernel and are well-behaved on circular domains.  In the next section, we'll explore some instances in which quadrature domains in several dimensions don't behave as well. 

\section{Several-Dimensional Counterexamples}

Our approach to higher-dimensional counterexamples will be to find biholomorphisms which map into quadrature domains in ways that are inaccessible in the plane.  We will find that the synergy between polynomials and the Bergman span of a quadrature domain is not as strong as in the plane.

We begin with a question on Bergman coordinates.  It has been noted here and in \cite{Bell5} that the fact in the plane that the image of a Bergman coordinate is a quadrature domain, may not hold in several variables.  We prove this now by writing down a biholomorphic mapping from the polydisc $\mathbb{D}^2 \subset \mathbb{C}^2$ which, although a Bergman coordinate, fails to map onto a quadrature domain.

\begin{prop}
	\label{BcornonQ}
	Let $f : \mathbb{D}^2 \rightarrow \mathbb{C}^2$ be defined by the formula $f(z_1,z_2)= (\frac{1}{3-z_1-z_2},z_1).$ This $f$ is a Bergman coordinate on $\mathbb{D}^2$, yet $f(\mathbb{D}^2)$ fails to be a quadrature domain.
	\begin{proof}
		We first appeal to Proposition \ref{QDcross} to see that $\mathbb{D}^2$ is a QDP. Hence every polynomial is in the Bergman span of $\mathbb{D}^2$, and each component of $f$ is a quotient of elements of the Bergman span.
		
		Since $3-z_1-z_2$ does not vanish on $\overline{\mathbb{D}^2}$, $f$ is holomorphic on $\mathbb{D}^2$ and leads to a bounded set $f(\mathbb{D}^2)$.  And if $f(z_1,z_2)=f(\zeta_1,\zeta_2)$, then from the second component function of $f$, $\zeta_1=z_1$. From the first component function, $3-z_1-z_2=3-\zeta_1-\zeta_2$, and it follows that $\zeta_2=z_2$.  Thus $f$ is a biholomorphism.
		
		We now see that $f(\mathbb{D}^2)$ is not a quadrature domain.  A simple calculation yields that the complex Jacobian determinant of $f$ is 
		\[u(z_1,z_2)= \frac{1}{(3-z_1-z_2)^2} \cdot 0 - \frac{1}{(3-z_1-z_2)^2}\cdot 1 = \frac{-1}{(3-z_1-z_2)^2}.\]
		We will show that $u$ cannot possibly be a member of the Bergman span of $\mathbb{D}^2$, and hence that the image of $f$ is not a quadrature domain.
		
		Any member of the Bergman span of $\mathbb{D}^2$ is a linear combination of terms of the form
		\begin{equation}
		\label{polyB}
		\frac{\partial ^\alpha K_{\mathbb{D}^2}(z,w)}{\partial \bar{w}^\alpha}|_{w=\omega}=\frac{\partial^{\alpha_1}K_{\mathbb{D}}(z_1, w_1)}{\partial \bar{w_1}^\alpha}|_{w_1=\omega_1} \cdot \frac{\partial ^{\alpha_2}K_{\mathbb{D}}(z_2,w_2)}{\partial \bar{w_2}^{\alpha_2}}|_{w_2=\omega_2}.
		\end{equation}
		Each of the factors on the right in (\ref{polyB}) is rational, and each depends solely on $z_1$ or on $z_2$. For contradiction we suppose that we can write $u$ as a linear combination of functions as in (\ref{polyB}). We would have \[\frac{-1}{(3-z_1-z_2)^2}= \sum_{j \le J} \frac{P_j(z_1)}{Q_j(z_1)}\cdot \frac{R_j(z_2)}{S_j(z_2)},  \]
		where each $P_j,Q_j,R_j,S_j$ is a polynomial of one variable. Clear denominators by multiplying each side by the product of all denominators appearing in the equation. The result is that:
		\[-\prod_{j \le J}Q_j(z_1)S_j(z_2)=(3-z_1-z_2)^2\sum_{j \le J} P_j(z_1)R_j(z_2)\prod_{k \ne j}Q_k(z_1)S_k(z_2) . \]
		Since the ring of holomorphic polynomials over any number of variables is a unique factorization domain, we can compare the two sides of this relation and notice a contradiction.  Each irreducible factor on the left side must be a polynomial of one variable depending solely on $z_1$ or $z_2$. Yet $3-z_1-z_2$ appears as an irreducible factor on the right side. So the left side contains only pure factors in $z_1$ and $z_2$, while there exists on the right side an irreducible factor which depends on both simultaneously.  This violates unique factorization.
\end{proof}
\end{prop}

The consequence is that, unlike in the plane, global Bergman coordinates cannot be used generically in several variables to map into quadrature domains in the attempt to show a kind of `generalized Riemann mapping theorem' \cite{Bell3}.  While in the plane the image of a biholomorphism is a quadrature domain if $u$ is in the Bergman span or if $f$ is a Bergman coordinate, in several variables the only avenue for mapping into quadrature domains is the condition on $u$.

There is in several variables a flexibility in choosing biholomorphic mappings on a domain. Having multiple variables allows us to save the property of one-to-oneness even in the presence of pathological components in a mapping, even while retaining a simple Jacobian determinant.  This is the approach to the next example, which shows that not every quadrature domain is a QDP in several variables. 

\begin{prop}
	Let $f(z_1,z_2)=(e^{z_1+z_2}+z_1,z_1+z_2)$ be defined on $\mathbb{D}^2$, the unit polydisc in $\mathbb{C}^2$.  Then $f$ is a biholomorphism, $V=f(\mathbb{D}^2)$ is a quadrature domain, and if $\zeta=(\zeta_1,\zeta_2)$ is the coordinate on $V$, then every monomial of the form $\zeta_2^k$ is in the Bergman span of $V$, whereas for $k>0$, $\zeta_1^k$ is not contained in the Bergman span of $V$.
	\begin{proof}
	We begin by noting that $f$ is biholomorphic.  If $f(z_1,z_2)=f(w_1,w_2)$, then from the second component of $f$, $z_1+z_2=w_1+w_2$, and so $e^{z_1+z_2}=e^{w_1+w_2}$.  Equating the first components of $f$ will therefore yield $z_1=w_1$, and once again equating the second components, this forces also $z_2=w_2$.
	
	To show that $V$ is a quadrature domain, we appeal to Theorem \ref{thm1} and calculate $u$, the complex Jacobian determinant of $f$. We find after differentiating that
	\[u(z_1,z_2)=(e^{z_1+z_2}+1) \cdot 1 - e^{z_1+z_2} \cdot 1=1.     \]
	We know by Proposition \ref{QDcross} that $\mathbb{D}^2$ is a QDP, so the function $1$, being a polynomial, is in the Bergman span of $\mathbb{D}^2$, and so by Theorem \ref{Jacobthm}, $V$ is a quadrature domain.
	
 Notice that for all $k$, $(z_1+z_2)^k$ is in the Bergman span of $\mathbb{D}^2$ since it is a QDP, and that $(z_1+z_2)^k=u \cdot f_2^k$. Lemma \ref{lem1} guarantees that on $V$, $\zeta_2^k$ is in the Bergman span.  Next, recall that all elements of the Bergman span of $\mathbb{D}^2$ are rational functions, since it has rational Bergman kernel.  Thus for all $k>0$, $(e^{z_1+z_2}+z_1)^k$ fails to be in the Bergman span of $\mathbb{D}^2$. By Lemma \ref{lem1} $\zeta_1^k$ fails to be in the Bergman span of $V$ for every $k \geq 1$.
		\end{proof}
	\end{prop}

The biholomorphism of the previous proposition also provides a counterexample to the algebraicity of the Bergman kernel of quadrature domains. In the plane, bounded quadrature domains have algebraic Bergman kernel \cite{Bell6}, so this is another loss of special properties in higher dimensions.

\begin{prop}
There exist bounded quadrature domains with non-algebraic Bergman kernel.

\begin{proof}
Let $f(z_1,z_2)=(e^{z_1+z_2}+z_1,z_1+z_2)$.It is one-to-one with inverse $F(\zeta_1, \zeta_2)=(\zeta_1-e^{\zeta_2},\zeta_2-\zeta_1+e^{\zeta_2}).$ Now $f$ is biholomorphic on all of $\mathbb{C}^2$ and has complex Jacobian determinant constantly equal to $1$.  By the transformation formula for the Bergman kernel, if $V=f(\mathbb{D}^2)$, $K_V(\zeta, \omega)=K_{\mathbb{D}^2}(F(\zeta),F(\omega))=\frac{1}{\pi^2}((1-(\zeta_1-e^{\zeta_2})\overline{(\omega_1-e^{\omega_2})})^{-2}(1-(\zeta_2-\zeta_1+e^{\zeta_2})(\overline{\omega_2-\omega_1+e^{\omega_2}}))^{-2}.$ 
\end{proof}
\end{prop}

We next address a question of Bell from \cite{Bell5}, whether the only $1$-point QDP's are constant-Jacobian images of circular domains. By an example we will demonstrate that it is possible to map a circular domain to a $1$-point quadrature domain by a mapping which does not have constant Jacobian. Consider the holomorphic mapping from $\mathbb{C}^2$ to itself: $f(z_1,z_2)=(z_1^2-z_2,z_1+z_2)$. The Jacobian determinant of $f$ is $u(z_1,z_2)=2z_1+1$. By the implicit function theorem, there is a neighborhood of $(0,0)$ on which $f$ is one-to-one. Within this neighborhood will be a polydisc $\Delta$ centered at $(0,0)$. By Proposition \ref{QDcross}, $\Delta$ is a QDP, which implies that $u$ is in the Bergman span of $\Delta$.  Even better, since $u$ and $f$ are polynomial, $uf^\alpha$ is polynomial for all multiindices $\alpha$, and so $f(\Delta)$ is a $1$-point QDP by Theorem \ref{CircThm}. The content of all this is the next proposition.

\begin{prop}
There exist $1$-point QDP's which are not the biholomorphic image of a circular domain under a constant-Jacobian biholomorphic map.
\begin{proof} Let $f$ and $\Delta$ be as given just before the statement of the current proposition. Then $f(\Delta)$ is a QDP, and an image of $\Delta$ under a map with non-constant Jacobian.  There is one technicality we must address. We must know that there is no other circular domain $\Omega$ and biholomorphism $g$ defined on $\Omega$ with constant Jacobian such that $g(\Omega)=f(\Delta).$ To exclude this possibility, we use the chain rule and Cartan's Uniqueness theorem, together with our knowledge of the automorphism group of a polydisc.

If such $\Omega$ and $g$ existed, then since $g(\Omega)=f(\Delta)$ contains $(0,0)$, we can find $\sigma$, an automorphism of $\Delta$ such that $\sigma(0,0)=f^{-1} \circ g (0,0)$. Then the mapping $g^{-1} \circ f \circ \sigma$ is a biholomorphism $\Delta \rightarrow \Omega$ which fixes the origin.  By Cartan's theorem, it is a linear mapping, which for convenience we will call $\lambda$. We have then $f=g \circ \lambda \circ \sigma^{-1}$. Since both $g$ and $\lambda$ have constant Jacobian determinants, the chain rule yields that the Jacobian determinant $u$ of $f$ is $u(z_1,z_2)=c s(z_1,z_2)$, where $s$ is the Jacobian determinant of $\sigma^{-1}$ and $c$ is a constant. But since $\sigma^{-1}$ is an automorphism of $\Delta$, both of its component functions are linear fractional transforms of $z_1$ or $z_2$.  For some constants $a_j, b_j, c_j, d_j$, $j=1,2$, we will have that for some constant $k$, \[u=c \cdot (\prod_{j=1,2}\frac{\partial}{\partial z_j}\frac{a_jz_j+b_j}{c_jz_j+d_j}) = \frac{k}{(c_1z_1+d_1)^2(c_2z_2+d_2)^2}. \]
But regardless the coefficients it is clear that this expression will never be a degree-one polynomial. This contradicts the formula for $u$ found just before the statement of the proposition.
\end{proof}
\end{prop}

We have employed the definition of $1$-point QDP from the previous section. It may still be that the only $1$-point QDP's which furthermore have a quadrature identity of one term are constant-Jacobian images of circular domains-to this question we do not provide an answer.

\section{Quadrature Domain Density}

We have seen that certain desirable properties of quadrature domains in the plane do not extend to the case of several dimensions, but in this section we examine a property which conceivably stands a good chance of passing to higher dimensions; namely, smooth density of quadrature domains among $\mathcal{C}^\infty$ smooth domains.

Finding density of quadrature domains in $\mathbb{C}^n$ could be viewed as establishing a `substitute Riemann mapping theorem' in several variables. Instead of finding the disc as a biholomorphy type for simply connected planar domains, we would be finding that some class of multidimensional domains are biholomorphically equivalent to `generalized discs', the generalization being the substitution of the mean-value property by a quadrature identity.

An approach to the problem of constructing biholomorphisms to quadrature domains with the added prospect of closeness to the identity function was proposed in \cite{Bell5}, and given explicit form for the case of the ball.  In this section, we generalize that argument to discover a class of domains which admit biholomorphisms to quadrature domains. Among this class will be all smooth bounded convex domains, and for some cases, we will be able to further infer that the domains may be mapped to quadrature domains $\mathcal{C}^\infty$ nearby.  A crucial piece of the approach is Condition R.

\begin{definition}
A bounded $\mathcal{C}^\infty$ smooth domain $\Omega \subset \mathbb{C}^n$ satisfies Condition R if the Bergman projection of $\Omega$ maps $\mathcal{C}^\infty (\overline{\Omega})$ into itself.
\end{definition}
In \cite{BL}, it is revealed that Condition R implies density of the Bergman span in $\mathcal{C}^\infty(\overline{\Omega})$ (though the words `Bergman span' are not used there). This density of Bergman span elements will be essential to the success of the ideas below.

Roughly speaking, a smooth bounded domain with convex cross-sections in the $z_n$ direction will be biholomorphic to a quadrature domain, as long as the domain has a pseudoconvex shadow onto the other $n-1$ dimensions, and the domain contains within it the graph of a smooth function over that shadow.

  \begin{thm}
  	\label{convexmain}
  	Let  $ \Omega  \subset \mathbb{C} ^n$ be a $\mathcal{C} ^{\infty} $ smooth bounded domain satisfying Condition R such that the projection $ \Pi _{n-1} \Omega $ of $\Omega$ to the first $n-1$ coordinates is pseudoconvex, and such that for each $z' \in \Pi _{n-1} \Omega$, the $z_n$-cross-section $ \{\tau \in \mathbb{C} \thinspace | \thinspace (z',\tau) \in \Omega \} $ is convex.  Assume furthermore that there exists $ \gamma \in \mathcal{C} ^ {\infty} (\Pi _{n-1} \Omega)$ with the property that for each $z' \in \Pi _{n-1} \Omega$, the point  $ (z', \gamma(z')) \in \Omega $.  Then there exists a biholomorphism $f$ defined on $\Omega$ such that $f(\Omega)$ is a quadrature domain.
  \end{thm}
  \begin{proof}
  	Since $\Omega$ satisfies Condition R, the Bergman span of $\Omega$ is dense in $\mathcal{C} ^{\infty} (\overline{\Omega} )$, (see \cite{BL}) and so we may choose $g$ in the Bergman span such that $g$ is close enough to $1$ in $\mathcal{C} ^{\infty} (\overline{\Omega} )$ that the function $U(z',z_n)= \int_{\gamma(z')}^{z_n} g(z', \tau)d \tau$ is one-to-one in $z_n$ for each fixed $z' \in \Pi _{n-1} \Omega$.  The function $U$ is well-defined since the $z_n$-cross-sections of $\Omega$ are simply connected. Note that for each fixed $z'$, $U$ is a holomorphic $z_n$-antiderivative of $g$; we modify $U$ in the first $n-1$ variables in order to obtain a function holomorphic on $\Omega$.
  	
  	Notice first that for $j = 1, 2, ..., n-1$, $ \frac{ \partial U}{\partial \bar{z}_j }$ is constant in the $z_n$ variable.   For, \[ \frac{\partial}{\partial z_n} \frac{\partial}{\partial \bar{z}_j}U=\frac{\partial}{\partial \bar{z}_j}\frac{\partial U}{\partial z_n}=\frac{\partial g}{\partial \bar{z}_j}=0 \]
  	and also \[ \frac{\partial}{\partial \bar{z}_n} \frac{\partial}{\partial \bar{z}_j}U=\frac{\partial}{\partial \bar{z}_j}\frac{\partial U}{\partial \bar{z}_n}=0. \]

  	Thus, for each $j=1,2,...,n-1$, we may regard $\frac{\partial U}{\partial \bar{z}_j}$ as a function on $\Pi_{n-1}\Omega$; which is to say there is a smooth function $\alpha_j$ on $ \Pi _{n-1} \Omega $ such that for every $z'$, $\alpha_j (z')= \frac{\partial}{\partial \bar{z}_j}U (z', z_n)$, for any $z_n$ such that $(z', z_n) \in \Omega$.  Define the smooth $(0,1)$-form $\alpha$ to be \[ \alpha=\sum_{j=0}^{n-1}\alpha_j d \bar{z}_j. \]  Then $\bar{\partial}\alpha=0$. To see this, let $z'$ be a given point of $ \Pi _{n-1} \Omega $, and let $\lambda$ be a complex number such that $(\zeta,\lambda)\in\Omega$ for all $\zeta$ nearby $z'$. Then for each such $\zeta$,\[\frac{\partial}{\partial \bar{z}_k}\alpha_j(\zeta)=\frac{\partial}{\partial \bar{z}_k}\frac{\partial}{\partial \bar{z}_j}U (\zeta, \lambda) =\frac{\partial}{\partial \bar{z}_j}\frac{\partial}{\partial \bar{z}_k}U (\zeta, \lambda)=\frac{\partial}{\partial \bar{z}_j}\alpha_k(\zeta).\]

  	Since $\Pi _{n-1} \Omega$ is pseudoconvex, we may choose $c \in \mathcal{C} ^{\infty}(\Pi _{n-1} \Omega)$ such that $\bar{\partial}c=\alpha.$  Define the smooth function $C$ on $\Omega$ by extending $c$ to be constant in the $z_n$ variable: $C(z',z_n)=c(z')$.  We then have that $U-C$ is holomorphic on $\Omega$.
  	
  	We now define the holomorphic map $f$ on $\Omega$ by \[ f(z',z_n)=(z',U(z',z_n)-C(z',z_n)). \]  Since $C$ is constant in the $z_n$ variable, a quick computation shows that the complex Jacobian determinant of $f$ is $g$.  Furthermore, $f$ is one-to-one on $\Omega$:  if $f(z', z_n)=f(\xi ', \xi _n),$ then $z'= \xi '$, and so $C(z',z_n)=C(\xi', \xi _n)$, which implies that $U(z', z_n)=U(\xi ', \xi _n)$, and thus $z_n$=$\xi_n$ since $U$ is one-to-one in the $n^{th}$ variable.
  	
  	So $f$ is biholomorphic, and since its complex Jacobian determinant  $g$ lies in the Bergman span of $\Omega$, $f(\Omega)$ is a quadrature domain by Theorem \ref{Jacobthm}.
  \end{proof}

  The hypotheses of the theorem apply in particular to all smooth bounded convex domains:

  \begin{thm}
  	\label{convcor}
  	If $\Omega \subset \mathbb{C}^n$ is a $\mathcal{C}^\infty$ smooth bounded convex domain, then there exists a biholomorphism $f$ defined on $\Omega$ such that $f(\Omega)$ is a quadrature domain.
  \end{thm}
  \begin{proof}
  	We need only to establish that $\Omega$ satisfies the conditions of Theorem \ref{convexmain}.  That the domain $\Omega$ satisfies Condition R is a consequence of the Sobolev estimates established in \cite{BoSt}.
  	
  	Since $\Omega$ is convex, it follows immediately that $\Pi_{n-1}\Omega$ is convex, and that for each fixed $z' \in  \Pi_{n-1}\Omega$, the set $\{\tau \in \mathbb{C} \thinspace | \thinspace (z',\thinspace \tau) \in \Omega \}$ is convex. In particular, $\Pi_{n-1}\Omega$ is pseudoconvex.
  	
  	At this point we need only to demonstrate the existence of an appropriate function $\gamma$. This piece of the proof we provide separately as a lemma.
  \end{proof}

  \begin{lem}
  	\label{convexsection}
  	If $\Omega$ is a $\mathcal{C}^\infty$ smooth bounded convex domain in $\mathbb{C}^n$, then there exists $\gamma \in \mathcal{C}^\infty (\Pi_{n-1}\Omega)$ such that for each $\zeta \in \Pi_{n-1}\Omega$, $(\zeta,\thinspace \gamma (\zeta)) \in \Omega$.
  \end{lem}
  \begin{proof}
  	Let $\Xi=\{\xi^{(\nu)}\}_{\nu=1}^\infty$ be a countable dense set of points of $\Omega$, and by way of notation say $\xi^{(\nu)}=(\xi^{(\nu)\prime},\thinspace \xi^{(\nu)}_n)$ in $(\mathbb{C}^{n-1} \times \mathbb{C})$-coordinates. For each $\nu$, define $B_\nu$ to be the maximal ball centered at $\xi^{(\nu)}$ contained in $\Omega$, $B_\nu=B(\xi^{(\nu)},\thinspace dist(\xi^{(\nu)},bd\Omega))$.  Since $\Xi$ is dense,  the collection $\{B_\nu \}_{\nu=1}^\infty$ is an open cover of $\Omega$.
  	
  	After projecting to the first $n-1$ coordinates, $\Pi_{n-1}B_\nu$ is the lower-dimensional ball $B(\xi^{(\nu)\prime},\thinspace dist(\xi^{(\nu)},bd\Omega))$, the ball in $\mathbb{C}^{n-1}$ centered at $\xi^{(\nu)\prime}$ with radius equal to that of $B_{\nu}$. So define $B_\nu'=\Pi_{n-1}B_\nu$, and we have that the collection $\{B_\nu'\}_{\nu=1}^\infty$ is an open cover of $\Pi_{n-1}\Omega$.
  	
  	Let $\{\varphi_\nu\}_{\nu=1}^\infty$ be a smooth partition of unity of $\Pi_{n-1} \Omega$ which is subordinate to $\{B_\nu'\}_{\nu=1}^\infty$. With this partition of unity, we are prepared to define the function $\gamma$.  For each $\zeta \in \Pi_{n-1}\Omega$, define $\gamma(\zeta)=\sum_{\nu}\varphi_\nu(\zeta)\cdot \xi^{(\nu)}_n$.
  	
  	The function $\gamma$ is smooth because it is locally the finite sum of smooth functions, and so we need only to check that, given $\zeta$, the point $(\zeta,\thinspace\gamma(\zeta))\in \Omega$.  We first notice that, since $\sum_{\nu}\varphi_\nu(\zeta)=1$, we have the equality $\sum_{\nu}\varphi_\nu(\zeta)\cdot \zeta = \zeta$.  We may thus write, in $(\mathbb{C}^{n-1}\times \mathbb{C})$-coordinates,
  	\begin{equation}
  	\label{convexcombo}
  	(\zeta,\thinspace \gamma(\zeta))=\left( \sum_{\nu}\varphi_\nu(\zeta)\cdot \zeta,\space \sum_{\nu}\varphi_\nu(\zeta)\cdot \xi^{(\nu)}_n \right)=\sum_{\nu}\varphi_\nu(\zeta) \cdot(\zeta,\thinspace \xi^{(\nu)}_n).
  	\end{equation}
  	Examining the terms which are non-vanishing, we reason that whenever $\varphi_\nu(\zeta)$ is not zero, we must have $\zeta \in B_\nu'$ by the subordination of the partition of unity. So for such $\nu$ we will have $|\zeta - \xi^{(\nu)\prime}|<dist(\xi^{(\nu)},bd\Omega)$, which implies that, by simply appending $\xi^{(\nu)}_n$ as an $n^{th}$ coordinate, $|(\zeta,\thinspace \xi^{(\nu)}_n)-\xi^{(\nu)}|<dist(\xi^{(\nu)},bd\Omega))$.  And this simply means that $(\zeta,\thinspace \xi^{(\nu)}_n) \in  \Omega$.
  	
  	So indeed (\ref{convexcombo}) expresses $(\zeta, \gamma(\zeta)))$ as a convex combination of points in $\Omega$, and the convexity of $\Omega$ ensures that $(\zeta,\thinspace \gamma(\zeta)) \in \Omega$.
  \end{proof}

  We have established that every smooth bounded convex domain is biholomorphic to a quadrature domain. But a closer look at the biholomorphism used in the proof reveals that in some cases, a biholomorphism  can be chosen which is $\mathcal{C^\infty}$ close to the identity.  If we suppose that $\gamma$ can be chosen holomorphic, then consider the following recipe for a mapping: let the first $n-1$ coordinates remain unchanged, and let the $n^{th}$ component be the integral $U(z',z_n)=\int_{\gamma(z')}^{z_n}g(z', \tau) d \tau$. By the chain rule $U$ is holomorphic, so there is no need to correct the integral with a $C$ to make it so. But $g$ was chosen to be $\mathcal{C}^\infty$ close to $1$. Thus, the $n^{th}$ component of the biholomorphism is $C^\infty$ close to $z_n-\gamma(z')$. Therefore, we modify the new mapping by adding $\gamma$ in the last coordinate, giving the mapping $f(z',z_n)=(z', U(z',z_n)+\gamma(z'))$. If the images of two points $(z',z_n)$ and $(\zeta', \zeta_n)$ are equal under $f$, then $z'=\zeta'$, so $\gamma(z')=\gamma(\zeta')$, and $U(z',z_n)=U(\zeta', \zeta_n)$. Since $U$ is one-to-one in the last variable when the others are held fixed, this shows that also $z_n=\zeta_n$, and $f$ is one-to-one.  Also, since $\gamma$ does not depend on $z_n$, the complex Jacobian determinant of $f$ is $g$, which is in the Bergman span. Thus the image of $f$ is a quadrature domain. The first $n-1$ coordinates remain unchanged by $f$, and we have that the last coordinate is $\mathcal{C}^\infty$ close to $z_n-\gamma(z')+\gamma(z')=z_n$. Thus $f$ is $\mathcal{C}^\infty$ close to the identity. All this proves the following corollary, concerning the possible density of quadrature domains with relation to convex domains.

  \begin{cor}
  Let $\Omega \subset \mathbb{C}^n$ be a smooth bounded convex domain such that a $\gamma$ as in the statement of Theorem \ref{convexmain} may be chosen holomorphic. Then, there exist quadrature domains arbitrarily $\mathcal{C}^\infty$ close to  $\Omega$ which are biholomorphically equivalent to $\Omega$.
  \end{cor}

  The corollary covers, as the most basic example, the case of smooth bounded convex domains which are symmetric about the hyperplane $\{z \in \mathbb{C}^n | z_n=0 \}.$ In this case, $\gamma$ may be taken to be $0$.

\section{Homotopies through Planar Quadrature Domains}

In this final section, we return to the plane to begin down a course of thought which may useful regarding the proposed idea of `quadular domains,' suggested in Bell \cite{Bell5}.  A quadular domain is defined there as a domain, all of whose $2$-dimensional cross-sections through a given point are quadrature domains.  It is the suggestion of \cite{Bell5} that such domains may be dense among many other domains.  This intuition is derived from the density of quadrature domains in the plane.

Of course, one could simply try to take each $2$-dimensional slice through a given point of a domain and transform it into a planar quadrature domain $\mathcal{C}^\infty$ nearby. The flaw here is that, from slice to slice, there is no continuity promised between the resulting quadrature domains; and hence the resulting set is generically not even an open set, just a `spiky' collection of slices arranged in space. To overcome this difficulty, it appears that some kind of homotopy is necessary.  While I do not resolve the issue of quadular domain density, I believe the results of this section could be a good first step in that direction.

Given two bounded simply connected planar quadrature domains, we will see that there exists a continuous deformation of one into the other such that each intermediate shape is a bounded simply connected quadrature domain.  From there, the result is generalized to the case of multiply connected domains in a local sense.  

Before we begin in earnest, we point out that the question of quadrature domain homotopy was taken up already by Sj\"{o}din in \cite{Sjodin}, but his result applies to quadrature domains for holomorphic $L^1$ functions, all of which admit the selfsame quadrature identity.  Here, as always, we work with quadrature domains for $H^2$ functions, and we will allow the quadrature identity to vary along the deformation.

The simply connected result can be obtained in an elementary manner, which we explain presently.  However, this intuitive approach will not suffice for multiply connected domains. Thus, after the initial result for simply connected domains, we will present a modified proof of the simply connected case; this modified version will apply only to simply connected quadrature domains with smooth boundary, but the desirable trade-off will be that the argument will generalize to the multiply connected case.

\begin{prop}
	\label{schomotopy}
Let $\mathbb{D}$ denote the unit disc, and let $\Omega$ be a bounded simply connected quadrature domain in $\mathbb{C}$ containing the origin. If $f$ is a conformal map $\mathbb{D} \rightarrow \Omega$ such that $f(0)=0$, then there is a homotopy $\varphi_t(z)$ with: $\varphi_0(z)=f'(0)z$, $\varphi_1(z)=f(z)$, and each $\varphi_t(\mathbb{D})$ is a quadrature domain.  (Note that this homotopy accomplishes a continuous deformation between $\Omega$ and a disc. By Riemann's mapping theorem, we conclude that any two bounded simply connected quadrature domains can be continuously deformed into one another in such a way that each intermediate shape is again a simply connected quadrature domain.)
	\begin{proof}
		For $(t,z) \in \bar{\mathbb{D}} \times \mathbb{D}$, set $\varphi(t,z)=\varphi_t(z)=\frac{f(tz)}{t}$ if $t \neq 0$ and $\varphi_0(z)=f'(0)z$. Writing the limit definition of $\frac{\partial}{\partial t}(f(z \cdot t))|_{t=0}$ shows that this definition is continuous even up to $t=0$. (An overpowered way to see this would be to note that our $\varphi$ is holomorphic in each variable separately, and so is holomorphic.) In particular, $\varphi$ is continuous on $[0,1] \times \mathbb{D}$. Furthermore, $\varphi_1(z)=f(z)$ for all $z$.

Our proof will be complete after some bookkeeping is accounted for.  First, for each fixed $t$, $0 \leq t \leq 1$, $\varphi_t(z)$ is biholomorphic.  For $t=0$, the formula makes it clear. For $t \neq 0$, suppose $\frac{1}{t}f(tz)=\frac{1}{t}f(tw)$ for $z,w \in \mathbb{D}$. Since $|tz|,|tw| < 1$, the fact that $f$ is one-to-one on $\mathbb{D}$ ensures that $tz=tw$, and $z=w$.

Finally, recall that the Bergman kernel of $\mathbb{D}$ is rational, so all elements of its Bergman span are rational, which means that $f$ is rational, since it is a mapping to a quadrature domain and must be a Bergman coordinate. But a glance at the definition of $\varphi$ thus reveals that for each fixed $t$, $\varphi(t,z)$ is also rational, and hence a Bergman coordinate. That means that $\varphi_t(\mathbb{D})$ is a quadrature domain for each $t$, $0 \leq t \leq 1$.
	\end{proof}
\end{prop}

In the course of the proof, the Riemann mapping theorem played the crucial role. This of course will not do for multiply connected domains, and so if we hope to generalize we need a different way of finding a homotopy. We will now once again show that simply connected quadrature domains are deformable to one another through quadrature domains. But this time, the essential component of the proof will be the modification of a `straight-line' homotopy. The idea of working with a straight-line homotopy will carry over to multiply connected domains, but we will require smoothness of the boundary; in the multiply connected case, we will also be restricted to a local result. 

 {\it Remark.} The reader is advised to glance ahead to Lemma \ref{lem5} before proceeding. It is valid for the unit disc (and in that case is essentially Theorem $1$ of \cite{ABG}), and we will be using it in that context. We include Lemma \ref{lem5} there, as opposed to here, in order to give continuity of presentation to the multiply-connected case.

\begin{prop}
Let $\mathbb{D}$ be the unit disc, and let $\Omega \subset \mathbb{C}$ be a bounded simply connected quadrature domain which is the image of a rational biholomorphic function $R:\mathbb{D} \rightarrow \Omega$ with the properties that $R(z)$ extends to be univalent in a neighborhood of $\overline{\mathbb{D}}$, and $R'(0)$ is not a real number.  Then, there exists a homotopy $\varphi_t(z)$, $0 \leq t \leq 1$, between the identity map on $\mathbb{D}$ and $R(z)$ such that each $\varphi_t$ is univalent on $\mathbb{D}$, and the image of $\mathbb{D}$ under each $\varphi_t$ is a quadrature domain. Such homotopy may be chosen to be a continuous rescaling of the straight-line homotopy.
\end{prop}
\begin{proof}
As the statement implies, the method is to modify the straight-line homotopy, so we begin with that homotopy, defined as: $f_t(z)=(1-t)z+tR(z)$, for $z\in \mathbb{D}$, $0 \leq t \leq 1$.  Notice that for $t,s \in [0,1]$, $f'_t(z)-f'_s(z)=(t-s)(R'(z)-1)$. Since $R$ extends to be analytic in a neighborhood of $\overline{\mathbb{D}}$, $R'$ is bounded on $\mathbb{D}$. Thus $|f'_t-f'_s|$ can be made uniformly small on $\mathbb{D}$ by taking $t$ and $s$ sufficiently close.  In particular, since both the identity map $f_0$ and the map $R=f_1$ extend to be univalent in a neighborhood of $\overline{\mathbb{D}}$, we can use Lemma \ref{lem5} below to see that for all $t$ sufficiently small, $f_t$ is univalent on $\mathbb{D}$, and for all $t$ sufficiently near to $1$, $f_t$ is univalent on $\mathbb{D}$.

Now define the real-valued function $r(t)$, $0 \leq t \leq 1$ by the formula:
 \[r(t)=\sup \{\rho \leq 1 \thinspace | \thinspace f_t \thinspace \thinspace \text{is univalent on} \thinspace \thinspace D_{\rho} \},\]
  where $D_{\rho}$ is the disc of radius $\rho$ centered at the origin.  We determine that $r(t)>0$ for each $t \in [0,1]$.  To see this, differentiate $f_t$ and set it equal to $0$, then solve for $R'(0)$ to see that $f'_t(0)=0$ if and only if $R'(0)=1-\frac{1}{t}$.  But since $R'$ is bounded and $R'(0)$ is not real, this is not the case. That means $f_t$ is biholomorphic in a neighborhood of the origin.

Actually, $r(t)$ turns out to be lower-semicontinuous on the interval $[0,1]$. Fix any value $t_0$ in the interval. Given small $\epsilon > 0$, let $\rho_0$ be a positive number such that $0<r(t_0)-\epsilon < \rho_0 < r(t_0)$.  Appealing to the definition of $r(t)$, $f_{t_0}$ is univalent in a neighborhood of $\overline{D_{\rho_0}}$.  Thus by Lemma \ref{lem5}, $f_t$ is univalent on $D_{\rho_0}$ for all $t$ sufficiently close to $t_0$.  Hence $r(t) \geq \rho_0 > r(t_0)-\epsilon$ for all $t$ sufficiently close to $t_0$.  And this is exactly what it means for $r(t)$ to be lower-semicontinuous at $t_0$. Since $t_0$ was arbitrary, $r(t)$ is lower-semicontinuous on all of $[0,1]$.

Since $r(t)$ is lower-semicontinuous on a compact set, it attains a miminum value $m$.  Since $r(t)$ is always greater than $0$, so is $m$.  To review, we have so far determined that $r(t)$ is equal to $1$ for all $t$ near $0$ and for all $t$ near $1$, and $r(t)$ has minimum value $m >0$. Our next step is to construct a continuous function on [0,1] which is equal to $1$ at the endpoints, and always greater than $0$ and at most $r(t)$. This function will provide a scaling which we can insert into our straight-line homotopy to ensure univalence at each value of $t$ in the homotopy.

Let $t_1$ and $t_2$ be such that $0<t_1<t_2<1$, and such that $f_t$ is univalent on $\mathbb{D}$ for all $t \in [0,t_1] \cup [t_2,1].$ We define the continuous function $k(t)$ to be that function which is equal to $1$ on $[0,\frac{t_1}{2}] \cup [\frac{t_2+1}{2},1]$; equal to $m/2$ on $[t_1,t_2]$; equal to $-\frac{2-m}{t_1}(t-\frac{t_1}{2})+1$ on $[\frac{t_1}{2},t_1]$; and equal to $\frac{2-m}{1-t_2}(t-\frac{1+t_2}{2})+1$ on $[t_2,\frac{1+t_2}{2}]$. Thus $k(t)$ has the property that $f_t$ is univalent on $D_{k(t)}$ for all $t \in [0,1]$.  Indeed, by construction $0<k(t) \leq r(t)$ for all $t$.

We are ready to define the homotopy $\varphi_t(z)$. It is a simple matter to see that, since $f_t$ is univalent on $D_{k(t)}$ for all $t \in [0,1]$, it is true that $f_t(k(t) \cdot z)$ is univalent on $\mathbb{D}$ for all $t \in [0,1]$. We define our homotopy to be:
\[\varphi_t(z)=f_t(k(t) \cdot z)=(1-t)k(t) \cdot z+tR(k(t) \cdot z).\]

Continuity in $t$ and $z$ simultaneously is clear, and for each fixed $t$, $\varphi_t$ is a rational function of $z$ which is univalent on $\mathbb{D}$.  Hence $\varphi_t$ is a Bergman coordinate, and the image of $\mathbb{D}$ under $\varphi_t$ is a quadrature domain.
\end{proof}

We note that if $\Omega$ is as in the proposition, except that $R$ has $R'(0)$ real, then the requirement that $R'(0)$ not be real may be affected by multiplying by a number $e^{i \theta}$, $\theta$ real.  This amounts to first continuously rotating $\Omega$ in the plane, and then applying the lemma. Appending the continuous rotation to the modified straight-line homotopy is again a homotopy.

It is the purpose of the following three lemmas to establish the necessary information to generalize to multiply connected domains. We recall that a domain $\Omega \subset \mathbb{C}$ satisfies a chord-arc condition if there exists a positive number $M$ such that, for any two points $z, \zeta \in \Omega$, a path $\gamma$ from $z$ to $\zeta$ contained in $\Omega$ may be found such that the length of the path satisfies $length(\gamma) \le M |z - \zeta|.$ Such an $M$ is called a chord-arc ratio for $\Omega$. The topic of relationships between the chord-arc condition and univalence of holomorphic functions is explored in the article \cite{ABG} (exclusively for simply connected domains); part of the proof of Theorem $1$ from \cite{ABG} inspires the proof of Lemma \ref{lem5}. 

\begin{lem}
	\label{lem3}
	Let $\Omega \subset \mathbb{C}$ be a bounded smooth finitely-connected domain such that each of the boundary curves of $\Omega$ is a circle.  Then $\Omega$ satisfies a chord-arc condition.
	\begin{proof}
		The hypotheses imply that the boundary circles are finitely many and mutually disjoint.  For simplicity, first assume that $bd \Omega$ contains two components, an outer circle and an inner circle.  Let $D=\{z \in \mathbb{C} : \thinspace |z-z_0|<r \}$ be the disc centered at $z_0$ of radius $r$, whose closure is the bounded component of the complement of $\Omega$.
		
		Let $z$ and $w$ be any two points of $\Omega$, and let $L(t)$ be the directed line segment from $z$ to $w$, $L(t)=(1-t)z+tw$, $0 \le t \le 1$.  If the line segment $L$ is contained in $\Omega$, then $length(L)=|z-w|$.  If $L$ is not contained in $\Omega$, then we see that the portion of $L$ not contained in $\Omega$ is either a single point on $bdD$ (in the case that $L$ lies tangent to $D$), or a chord of $\overline{D}$. In either case, since $z$ and $w$ lie in $\Omega$, their distances to $\overline{D}$ are positive, and so we may choose two points $z^*$ and $w^*$ on $L$ such that, if we say $z^*=L(t_0)$ and $w^*=L(t_1)$, we have $t_0 < \inf\{t: \thinspace L(t) \in \overline{D} \}$ and $t_1 > \sup\{t: L(t) \in \overline{D}\}$, and $z^*,w^*$ each have distance exactly $r + \epsilon >0$ to $z_0$, where $\epsilon$ is chosen less than the distance between the boundary curves of $\Omega$, and such that $z \ne z^*$ and $w \ne w^*$. (More informally, enlarge $bdD$ a bit to radius $r+\epsilon$, and find where the enlarged circle intersects $L$; there are two intersection points, $z^*$ and $w^*$.)
		
		Now, let $L_1(t)$ denote $L(t)|_{0 \le t \le t_0}$ and let $L_2(t)$ denote $L(t)|_{t_1 \le t \le 1}$. Let $\gamma(t)$, $t_0 \le t \le t_1$ parameterize the shorter circular arc of radius $r +\epsilon$ about $z_0$, starting from $z^*$ and ending at $w^*$.  The length of this arc is at most $\frac{\pi}{2}|z^*-w^*|$.
		
		Finally, define $\Gamma= L_1 \# \gamma \# L_2$.  Thus $\Gamma$ is the original $L$, except with a circular deformation from $z^*$ to $w^*$.  The curve $\Gamma$ lies entirely within $\Omega$, and its length is
		\begin{multline}
		length(\Gamma)=length(L_1)+length(\gamma) + length(L_2) \\ =  |z-z^*| + length(\gamma) + |w-w^*|   \le \frac{\pi}{2}(|z-z^*| + |z^*-w^*| + |w^*-w|) = \frac{\pi}{2}|z-w|.
		\end{multline}
		
		We have used the fact that $z,z^*,w^*,w$ are collinear.  This establishes the $1$-connected case.
		
		For the general case, repeat the above; perform an analogous circular deformation around the successive components of the complement of $\Omega$ which the line segment from $z$ to $w$ may intersect. The same chord-arc ratio of $\frac{\pi}{2}$ will be found.
		\end{proof}
\end{lem}

\begin{lem}
	\label{lem4}
If $\Omega \subset \mathbb{C}$ is a bounded $\mathcal{C}^\infty$ smooth finitely-connected domain, then $\Omega$ satisfies a chord-arc condition.
\begin{proof}
	By a result in conformal mapping, planar circle domains are a canonical class for biholomorphisms of multiply connected domains. That means there is a biholomorphism $f$ defined on $\Omega$ such that $f(\Omega)$ is a bounded smooth finitely-connected domain such that all of its boundary curves are circles  (see \cite{Goluzin} for an exposition).  By a well-known theorem of Painlev\'{e}, this $f$ extends $\mathcal{C}^\infty$ smoothly up to the boundary of $\Omega$.
	
	Given two distinct points $z,w \in \Omega$, consider their images $f(z),f(w)$. By the previous lemma, there exists a path $\gamma(t), 0 \le t \le 1$, with $\gamma(0)=f(z)$, $\gamma(1)=f(w)$, and $length(\gamma) \le \frac{\pi}{2}|f(z)-f(w)|.$  Now, since $f$ extends smoothly to the boundary of $\Omega$, it even extends smoothly (though not necessarily holomorphically) to a neighborhood of $\overline{\Omega}$, and thus is Lipschitz on $\Omega$, so that there exists a constant $M$ such that $|f(\zeta)-f(\omega)| \le M|\zeta - \omega|$, valid for all $\zeta, \omega \in \Omega$.
	
	We have so far determined that $length(\gamma) \le \frac{\pi}{2}M|z-w|.$  We pull back $\gamma$ to $\Omega$, defining $\Gamma(t)=F \circ \gamma (t), 0 \le t \le 1$, where $F$ is the inverse of $f$.  We have $\Gamma(0)=z, \Gamma(1)=w$, and the length of $\Gamma$ can be calculated as follows:
	
	\begin{multline} length(\Gamma)= \int_0^1 |(F \circ \gamma)'(t)|dt = \int_0^1 |F'(\gamma(t)) \cdot \gamma '(t)| dt \\ \le C \int_0^1 |\gamma '(t)| dt  = C \cdot length(\gamma) \le CM \frac{\pi}{2}|z-w| ,
	\end{multline}
	where $C$ is a bound of $|F'|$ on $f(\Omega)$ ($F'$ is bounded because $F$ extends smoothly to the boundary of $f(\Omega)$. Thus $\Omega$ exhibits a chord-arc condition.	
\end{proof}	
\end{lem}

\begin{lem}
	\label{lem5}
Suppose $\Omega \subset \mathbb{C}$ is a bounded $\mathcal{C}^\infty$ smooth finitely-connected domain, and let $f$ be a holomorphic function which is univalent in a neighborhood of $\overline{\Omega}$.  Then, there exists $\epsilon > 0$ such that, whenever $g$ is a holomorphic function on $\Omega$ and $\sup_{z\in \Omega}{|f'(z)-g'(z)|} < \epsilon$, it follows that $g$ is also univalent on $\Omega$.

\begin{proof}
Let $M>0$ be a chord-arc ratio for $\Omega$ , and let $m>0$ be such that for all $z,w \in \Omega$, $|f(z)-f(w)| \geq m|z-w|$.  That such an $m$ exists is justified in the following way.  Define $F(z,w)$ on $\Omega \times \Omega$ with $F(z,w)=\frac{f(z)-f(w)}{z-w}$ if $z \neq w$ and $F(z,w)=f'(z)$ if $z=w$. Then $F$ is holomorphic in each variable separately, and so is holomorphic on $\Omega \times \Omega$ by Hartogs's Theorem. Since $f$ is univalent on a neighborhood of $\overline{\Omega}$, we know that when $z \neq w$, $f(z)-f(w) \neq 0$. We also know that $f'(z)$ is nonvanishing on $\overline{\Omega}$.  This all goes to show that $F(z,w)$ is continuous and non-vanishing on $\overline{\Omega} \times \overline{\Omega}$. Thus $|F|$ obtains a positive minimum value $m$ on $\overline{\Omega} \times \overline{\Omega}$, and this is the $m$ we meant to find.

Let $\epsilon < \frac{m}{M}$. We shall see that this $\epsilon$ satisfies the Lemma.  So let $g$ be any holomorphic function on $\Omega$ such that $|g'-f'| < \epsilon$ on $\Omega$. To show that $g$ is univalent, we will establish that for $z \neq w$, the value $|g(z)-g(w)|$ cannot possibly be $0$. Given distinct $z,w$ in $\Omega$, let $\gamma$ be a curve in $\Omega$ from $w$ to $z$ such that $length(\gamma) \leq M|z-w|$. Then write:
 \[g(z)-g(w) = \int_\gamma (g'(\zeta)-f'(\zeta))d\zeta + \int_\gamma f'(\zeta)d\zeta = \int_\gamma (g'(\zeta)-f'(\zeta))d\zeta + f(z)-f(w). \]

 Consider now that $|f(z)-f(w)| \geq m|z-w|$, whereas
 \[|\int_\gamma (g'(\zeta) - f'(\zeta))d\zeta| \leq \epsilon \cdot length(\gamma) \leq \epsilon \cdot M \cdot |z-w| < m|z-w| .\] Hence $f(z)-f(w)$ and $\int_\gamma (g'(\zeta)-f'(\zeta))d\zeta$ have no chance of summing to $0$, which in turn means that $g(z)-g(w)$ is non-zero.
\end{proof}	
\end{lem}

Finally we are in a position to string these lemmas into a theorem concerning deformation of multiply-connected quadrature domains.

\begin{thm}
Let $\Omega \subset \mathbb{C}$ be a $\mathcal{C}^\infty$ smooth bounded finitely-connected domain, and let $f$ be a function which is univalent on a neighborhood of $\overline{\Omega}$ and such that $f(\Omega)$ is a quadrature domain.  Then there exists $\epsilon >0$ such that whenever $g$ is a holomorphic function on $\Omega$ such that $\sup_\Omega|g'-f'| < \epsilon$ and $g(\Omega)$ is a quadrature domain, the following holds: if $\varphi_t(z)=(1-t)f+tg$, $0 \leq t \leq 1$ is the straight-line homotopy from $f$ to $g$, then each $\varphi_t$ is biholomorphic on $\Omega$, and each $\varphi_t(\Omega)$ is a quadrature domain.
\begin{proof}
Choose $\epsilon$ satisfying Lemma \ref{lem5}, and let $g$ be as in the statement of the theorem according to this $\epsilon$.  Then for each $t$ in $[0,1]$ and $z \in \Omega$, we will have $|\varphi_t'(z)-f'(z)|=|(1-t)f'(z)+tg'(z)-f'(z)|=t|f'(z)-g'(z)| < \epsilon.$ By Lemma \ref{lem5}, then, $\varphi_t$ is univalent on $\Omega$.  Since $f$ and $g$ both map $\Omega$ into quadrature domains, it follows that $f'$ and $g'$ are members of the Bergman span of $\Omega$. Since the Bergman span is a linear space and each $\varphi_t'$ is a linear combination of $f'$ and $g'$, it follows that each $\varphi_t'$ is in the Bergman span of $\Omega$.  Hence, $\varphi_t(\Omega)$ is a quadrature domain for each $t$, and the theorem is proved.
\end{proof}
\end{thm}

The following corollary to the theorem is the promised result spelling out circumstances under which a quadrature domain can be continuously deformed into a nearby quadrature domain.

\begin{cor}
If $\Omega$ is a bounded $\mathcal{C}^\infty$ smooth finitely-connected quadrature domain, let $\epsilon$ satisfy the hypotheses of the theorem using the identity map $f(z)=z$.  Then, if $V$ is a quadrature domain which is conformally equivalent to $\Omega$ under a biholomorphic map $g: \Omega \rightarrow V$ which is within $\epsilon$ of the identity in the $\mathcal{C}^1(\Omega)$ norm, then $\Omega$ may be continuously deformed into $V$ in such a way that each intermediate domain is a quadrature domain and conformally equivalent to $\Omega$.

\begin{proof} It is easy to see that $f(z)=z$ is satisfactory for the implementation of the theorem, and then of course $\varphi_0(\Omega)=\Omega$, $\varphi_1(\Omega)=V$, and each $\varphi_t(\Omega)$ is a quadrature domain.
\end{proof}
\end{cor}

These results, aside from their possible utility in justifying the notion of quadular domain, are a further testament to the profusion of quadrature domains to be found in the plane. More than being dense among smooth domains, we can now appreciate that quadrature domains can even be continuously deformed into one another through other quadrature domains.

\bibliography{all}
\bibliographystyle{plain}

\end{document}